\documentclass[11pt,reqno]{amsart}

 \usepackage{graphicx}
 \usepackage{amscd,amssymb}

 \def\al{\alpha}
 
 \def\de{\delta}
 \def\eps{\varepsilon}
 
 \def\ga{\gamma}

 \def\La{\Lambda}
 \def\om{\omega}

 \def\Laij{{\La_{ij}}}

 \def\R{{\mathbb R}}
 \def\Z{{\mathbb Z}}
 \def\N{{\mathbb N}}
 \def\T{{\mathbb T}}
 \def\I{{\mathbb I}}
 \def\E{{\mathbb E}}
 \def\Pr{{\mathbb P}}

\def \DiM{{ \mathcal D \big (0,+ \infty;\{1, \cdots, M\} \big )}}
\def \DiR{{ \mathcal D \big ( 0,+ \infty;\mathbb R^N \big )}}
\def \CiT{{ \mathcal C \big ( 0,+ \infty;\mathbb T^N \big )}}
\def \S{{ \mathcal S}}
\def \K{{ \mathcal K}}
\def \D{{ \mathcal D}}

\def\a{\mathbf a}
 \def\ai{a_i}
 \def\b{{\mathbf b}}
 
 \def\d{{\mathbf d}}

 \def\ei{{ \mathbf e}_i}

 \def\F{{\mathcal F}}

 \def\A{{\mathcal A}}
 \def\uno{{\mathbf 1}}

\def\L{{\mathbf L}}
 
 \def\cH{{\check{\mathbf H}}}
 \def\cH_i{{\check{ H_i}}}
 
 \def\cH{{\check{ L_i}}}
 \def\u{{\mathbf u}}

 \def\v{{\mathbf v}}

 \def\w{{\mathbf w}}
 \def\wi{w_i}

 \newcommand{\mres}{\mathbin{\vrule height 1.6ex depth 0pt width
 0.13ex\vrule height 0.13ex depth 0pt width
 1.3ex}}

 \DeclareMathOperator{\essinf}{ess\hskip 1.5pt inf}
 \DeclareMathOperator{\unoM}{\{1, \cdots, M\}}

 \DeclareMathOperator{\proj}{proj}

 \DeclareMathOperator{\im}{im}

\renewcommand{\proofname}{{\bf Proof:}}
\theoremstyle{plain}

\swapnumbers

\newtheorem{Thm}{Theorem}[section]

\newtheorem{Lemma}[Thm]{\bf Lemma}

\newtheorem{Corollary}[Thm]{\bf Corollary}
\newtheorem{Theorem}[Thm]{\bf Theorem}
\newtheorem{Proposition}[Thm]{\bf Proposition}

\numberwithin{equation}{section}

\theoremstyle{definition}
\newtheorem{Definition}[Thm]{\bf Definition}
\theoremstyle{remark}
\newtheorem{Remark}[Thm]{\bf Remark}

\newtheorem{Terminology}[Thm]{\bf Terminology}
\newtheorem{Notation}[Thm]{\bf Notation}

 \newtheoremstyle{Cl}
  {5pt}
  {3pt}
  {\sl}
  {}
  {\it}
  {:}
  {.5em}
  {}

 \theoremstyle{Cl}

 \def\begincproof{
                  \renewcommand{\proofname}{\it Proof:}
                  \begin{proof}
                 }

 \def\endcproof{
                \renewcommand{\qedsymbol}{$\diamondsuit$}
                \end{proof}
                \renewcommand{\qedsymbol}{\openbox}
                \renewcommand{\proofname}{\bf Proof:}
               }

\renewcommand{\proofname}{{\bf Proof:}}

\title[A Lagrangian approach to weakly coupled HJ systems]
{A Lagrangian approach to  weakly coupled Hamilton--Jacobi systems}

\thanks{
The work of HM was partially supported by JST program to disseminate
tenure tracking system,  the work of AS was partially supported by
Programma Ricerca Scientifica Sapienza 2013, and the work of HT was partially supported by
NSF grant DMS-1361236.}

\author[H. Mitake]{H. Mitake}
\address[H. Mitake]{
Institute for Sustainable Sciences and Development,
Hiroshima University 1-4-1 Kagamiyama, Higashi-Hiroshima-shi 739-8527, Japan.}
\email{hiroyoshi-mitake@hiroshima-u.ac.jp}

\author[A. Siconolfi]{A. Siconolfi}
\address[A. Siconolfi]{Dipartimento di Matematica,
Universit\`a degli Studi di Roma ``La Sapienza'' 00185 Roma, Italy.}
\email{siconolf@mat.uniroma1.it}

\author[H. V. Tran]{H. V. Tran}
\address[H. V. Tran]
{Department of Mathematics, The University of Chicago, 5734 S. University Avenue, Chicago, Illinois 60637, USA.}
\email{hung@math.uchicago.edu}

\author[N. Yamada]{N. Yamada}
\address[N. Yamada]{Department of Applied Mathematics, Faculty of Science, Fukuoka University, Fukuoka 814-0180, Japan.}
\email{nyamada@math.sci.fukuoka-u.ac.jp}

\begin{document}

\begin{abstract}
We study a class of weakly coupled Hamilton--Jacobi systems
with a specific aim to perform
    a qualitative analysis in the spirit of weak KAM
   theory. Our main achievement is the definition of  a family of related  action functionals
   containing the Lagrangians obtained by duality from  the Hamiltonians of the system. We use them to characterize,
    by means
   of a suitable  estimate, all the subsolutions of the system,  and to explicitly represent
some  subsolutions enjoying an additional maximality property.
A crucial step  for our analysis is to put the problem in a suitable
random frame. Only some basic knowledge of measure
theory is required, and the presentation is accessible to readers without
background in probability.
\end{abstract}


 \maketitle

\section{Introduction}
\parskip +3pt

This paper deals  with weakly coupled Hamilton--Jacobi systems of the
form
\[ \tag{HJ$\al$}
\left\{\begin{array}{l}
H_1(x,Du_1) + \Lambda^1 \cdot \mathbf u = \alpha  \\
\cdots \\
 H_M(x,Du_M) + \Lambda^M \cdot \mathbf u = \alpha
\end{array}\right.
\]
on the flat torus $\T^N$. Here $\u=(u_1, \cdots, u_M)$ is the
vector valued unknown function,  $Du_i$ the gradient of $u_i$,
$\al$  a real number, and $H_i$ are mutually
unrelated  convex Hamiltonians enjoying standard additional
properties (see Section \ref{setting}). The $\Lambda^i$ are the rows
of the so called $M \times M$ coupling matrix $\La:=(\La^1 \cdots \La^M)$, which
constitutes the relevant item in  the problem.

We are specifically interested in the setting which should
correspond in the scalar case, namely when $M=1$ and  $\La$ is just
 a constant,  to taking $\La =0$.  Then the system reduces to  a
single equation on $\T^N$  not directly depending on the unknown and
 classified as  of Eikonal type.

In this framework  a rich qualitative theory has been developed by
linking  PDE facts to geometrical/dynamical properties.
Representation formulae for (sub)solutions have been provided
through minimization of a suitable action functional, showing, among
other things, the existence of an unique value of $\al$, named a
critical value, for which (viscosity) solutions do exist. This
material has found applications in a variety of related asymptotic
problems, and connections with Hamiltonian dynamics have been
furthermore established, at least when the Hamiltonian is
sufficiently regular. This body of results is a part of the
so-called weak KAM theory, see
 \cite{Be,DS, FathiBook,FS1,FS2,I} for
details.

We recall that if instead  $\Lambda >0$ the corresponding equation
can be uniquely solved on the whole torus for any $\alpha$ and the
solution is the value function of a related control problem with
$\Lambda$ playing the role of discount factor.

To find an analogue of the  Eikonal case for systems, it is
convenient to start from paper \cite{IK}, where the class of
monotone systems is introduced, and existence and uniqueness results
of (viscosity) solutions are established. Regarding our system, to
be a monotone one corresponds to the following conditions on the
coupling matrix:
\begin{itemize}
    \item any non--diagonal entry of $\La$ is nonpositive;
    \item $\La$ is diagonal dominant, namely  $\sum_{j=1}^M \Laij \geq 0$ for any
    $i\in\{1,\dots,M\}$;
    \item strict diagonal dominance holds at least for one row.
\end{itemize}

This setting should be then  analogized to strict positiveness in
the scalar case and in this perspective it is consistent to focus on
the limit setup where  $\La$ satisfies:
\begin{itemize}
    \item any non--diagonal entry of $\La$ is nonpositive;
    \item any row sums to $0$.
\end{itemize}

It has been actually a merit of \cite{CaLeLoNg, MiTr1, MiTr2, MiTr3}
to have first realized and pointed out that under the above
assumptions on the coupling matrix, some phenomena, already
occurring in the Eikonal scalar case, also take place for systems,
and can be analyzed in the spirit of the weak KAM theory. In these
papers it has been   in particular showed the existence of a
critical value as the minimal value for which the corresponding
system admits subsolutions,  and some related asymptotic problems
have been studied  providing generalization of results already known
in the scalar case. Control interpretation for the Hamilton--Jacobi
system has clearly been investigated in \cite{MiTr2, MiTr3}. We also
refer to \cite{GS} for the study of the weak KAM theorem of another
type of systems.

A significant step forward in this direction has been more recently
performed in \cite{DaZa}, proving that, similarly to what happens in
the scalar case,  a distinguished subset of the torus,  named after
Aubry, can be defined   with the crucial property that the maximal
critical subsolution (i.e., a subsolution to the system with $\al$
equal to the critical value) taking a given value, among admissible
ones, at any fixed point  of the Aubry set is indeed a critical
solution. The aforementioned admissibility refers to the fact that
there is a restriction in the values that a subsolution of the
system  can assume at any given point. This is a further relevant
property pointed out in \cite{DaZa}, which genuinely depends on the
vectorial structure of the problem and has no equivalent in the
scalar case.

All the above  results, even if of clear interest, however pertain
to the PDE side of the theory, and are solely obtained by means of
PDE techniques. The geometric counterpart is so far missed and   the
intertwining between PDE and dynamical aspects,  which is at the
core of the weak KAM theory, has consequently still to be understood in
the framework of systems. This is actually the primary task the
paper is centered upon,  and  is above all performed  by  putting
the problem in a suitable random frame.

As a first step we consider  all the possible switchings between
indices $\unoM$ of the system  on an infinite time horizon. This
gives rise to the space of $\unoM$--valued cadlag paths, denoted by
$\D$,  endowed with the Skorohod metric and the corresponding Borel
$\sigma$--algebra $\F$. The coupling matrix, being under our
assumptions generator of a semigroup of stochastic matrices, induces
a linear correspondence between the simplex of  probability vectors
of $\R^M$, i.e., with nonnegative components summing to $1$, and a
simplex of $\F$--probability measures on $\D$, see Subsection
\ref{family}.

This construction is indeed equivalent to that of a Markov  chain
with rate matrix $- \La$, and in fact key formula \eqref{newprobaa}
defining the family of probability measures is nothing but the usual
finite--dimensional distribution formula with given initial
distribution. However we would like to emphasize that the advantage
of our approach is to avoid introducing an abstract probability
space, we just work with concrete path spaces, and also avoid
explicitly using notions as stochastic process, conditional
probability and other probabilistic tools. This makes the
presentation self--contained.

We make corresponding to elements of $\D$ $\R^N$--valued cadlag
velocity paths and obtain by integration of it the admissible random
curves on $\T^N$, see Subsection  \ref{ammetti}. Action functionals
are then obtained by averaging, with respect to previously
introduced probability measures on $\D$, line integrals over random
curves on time random intervals of the Lagrangians given by duality
by the Hamiltonians of the system, see \eqref{action},  which
justifies the title of the paper.

The effectiveness of our approach is demonstrated  by recovering
some crucial facts of the scalar case. Namely, we fully characterize
 all subsolutions of the system, for any $\al$ greater than or
equal to the critical value,  as the functions from $\T^N$ to $\R^M$
satisfying a suitable estimate with respect to our action
functionals, see Section \ref{estimate} and Theorem \ref{susuvai}.
We moreover use the action functionals to represent explicitly
critical and supercritical subsolutions enjoying an additional
maximality property, through a suitable minimization procedure,  see
Theorem \ref{mainsub}, and to give a dynamical formulation of the
property of being admissible for a value at a given point, see
Theorem \ref{suvai}. By this way we also provide a representation
formula for critical solutions taking a prescribed admissible value
at a given point of the Aubry set, complementing the  result of
\cite{DaZa}, see Theorem \ref{vai}.

The paper is organized as follows: in Section \ref{setting} we set
forth  the problem and recall some known facts about
critical/supercritical subsolutions and the Aubry set.  Section
\ref{random} is devoted to illustrate the random frame in which our
qualitative analysis takes place: the family of probability measures
$\Pr_\a$, for any probability vector $\a$ of $\R^M$, is introduced
and key notions as admissible control and stopping time are given.
In Section \ref{estimate} we define the action functionals and prove
the fundamental estimate for subsolution to the system. Section
\ref{subsol} is about representation formulae for subsolutions and
related results. Finally the two appendices gather basic material on
stochastic matrices and spaces of cadlag paths.

\bigskip

\section{Setting of the problem}\label{setting}

\parskip +3pt

Here we introduce the system, which is the object of investigation,
as well as standing assumptions and basic preliminary facts.  We
refer to \cite{CaLeLoNg, DaZa, MiTr1,MiTr2} for proofs and more
details on the results stated.

As already pointed out in Introduction, we will be  interested
on the one--parameter family of systems (HJ$\al$)
\begin{equation}\label{HJa} \tag{HJ$\al$}
\left\{\begin{array}{l}
H_1(x,Du_1) + \Lambda^1 \cdot \mathbf u = \alpha  \\
\cdots \\
 H_M(x,Du_M) + \Lambda^M \cdot \mathbf u = \alpha
\end{array}\right.
\end{equation}
posed on the flat torus $\T^N$ identified to  $ \R^N/\Z^N$. Here
$\mathbf u=(u_1, \cdots, u_M)$ is the vector--valued unknown
function, $\Lambda^i$ are the vectors given by the rows of the $M
\times M$ {\em coupling matrix} $\Lambda$, and $\al$ varies in $\R$.
The following conditions will be assumed throughout the paper
without any further mentioning.  On Hamiltonians $H_i$ we require
\begin{itemize}
    \item[{\bf(H1)}] $H_i$ is continuous in both variables;
    \item[{\bf(H2)}] $H_i$ is convex in $p$;
    \item[{\bf(H3)}] $H_i$ is superlinear in $p$;
\end{itemize}
\smallskip
The growth condition in {\bf(H3)}, together with
{\bf(H1)}, {\bf(H2)}, allows defining   the  corresponding
Lagrangians via the Legendre--Fenchel transform, namely
\[L_i(x,q)= \max_{p\in\R^n} \big(p \cdot q - H_i(x,p)\big) \quad\hbox{for any $i$},\]
and they inherit from $H_i$  the properties of being continuous, convex
and superlinear at infinity.

We furthermore require on coupling matrix $\La$:
\begin{itemize}
\item[{\bf(H4)}] any non--diagonal entry  of $\La$ is nonpositive.
\item[{\bf(H5)}] any row of $\La$ sums to $0$.
\item[{\bf(H6)}] $\La$  is irreducible.
  \end{itemize}
{\em Irreducible} means that, given any nonempty subset of indices $I \subsetneq
\unoM$, there is $i \in I$, $j \not\in I$ with $\Laij \neq 0$;
loosely speaking this condition means that  the system cannot be
split in separated subsystems.

As made precise in Appendix \ref{uno}, the key point is that
{\bf(H4)}, {\bf(H5)} are equivalent to  $- \La$ being generator of a
semigroup of stochastic matrices. We also recall  that under
{\bf(H4)}, {\bf(H5)}, {\bf(H6)} the matrix $\La$ is singular with
rank  $M-1$ and kernel spanned by $\uno$, namely the vector with all
components equal to $1$, moreover $\im(\La)$ cannot contain vectors
with strictly positive or negative components. This in particular
implies $\im(\La) \cap \ker(\La)=\{0\}$.

\medskip

\begin{Notation}\label{notproje}  The projection of $\R^N$ onto $\T^N =
\R^N/\Z^N$ induces a structure of additive group on $\T^N$. To
ease notations we will indicate throughout the paper by the usual
symbols $+$, $-$ the corresponding operations between elements of
the torus.

\end{Notation}

\medskip

The notion of viscosity (sub/super)solution  can be easily adapted
to systems as \eqref{HJa}, we will drop in the following  the term
viscosity since no other kind of weak solution will be considered.

\smallskip

\begin{Definition} A continuous function $\u=(u_1, \cdots, u_M)$ is a
subsolution (resp., supersolution) of \eqref{HJa} if the inequality
\[ H_i(x,D\psi(x)) + \La^i \cdot \u(x) \leq \al \quad (\text{resp.},  \geq \al)\]
holds for every $x \in \T^N$,  $i \in \unoM$, and $\psi\in
C^1(\T^n)$ such that $u_i-\psi$ attains a maximum (resp., minimum)
at $x$. We call $\u$ a solution if it is both a subsolution and
supersolution.
\end{Definition}
\medskip

\begin{Remark} One could wonder why we are considering systems with the same constant appearing
in the right--hand side of any equation, while a more natural
condition should be to have instead a vector of $\R^M$, say $\a$,
with possibly different components. We point out that, under our
assumptions,  such a setting is actually no more general. In fact,
if we write the vector $\a$
 as $\a_1 + \a_2$ with $\a_1= \al \, \uno \in \ker(\La)$, $\a_2 \in \im(\La)$, where this form is uniquely
determined because $\im(\La) \cap \ker(\La) = \{0\}$, and pick $\b$
with $\La \, \b= - \a_2$,  then $\u$ is a (super/sub)solution to
\eqref{HJa} if and only if $\u + \b$ satisfies the same properties
for the system obtained from \eqref{HJa} by replacing in the right
hand side $\al \, \uno$ by $\a$.
\end{Remark}

\begin{Remark}\label{solu}
Due to the coercivity condition, any subsolution to \eqref{HJa}
is Lipschitz continuous. Moreover, owing to the convexity of the
Hamiltonians, the notion of viscosity and \textit{a.e.} subsolutions are
 equivalent for \eqref{HJa}. Furthermore, we can express the
same property using generalized gradients of any component in the
sense of Clarke. Namely, $\w$ is a subsolution to \eqref{HJa} if and
only if
\[ H_i(x,p) + \La^i \cdot \w(x) \leq \al\]
for any $x \in \T^N$, $p \in \partial \wi(x)$, $i \in \unoM$, where
$\partial \wi(x)$ indicates the generalized gradient of $\wi$ at
$x$.
\end{Remark}

\medskip

Here are two basic propositions.
\begin{Proposition}\label{preli} The family of all  subsolutions to \eqref{HJa}, if nonempty, is
equi-Lipschitz continuous with Lipschitz constant denoted by
$\ell_\al$.
\end{Proposition}

\medskip

\begin{Proposition}\label{maximal}
The family of  subsolutions to \eqref{HJa}
taking the same value at a given point, if nonempty, admits a
maximal element.
\end{Proposition}

\medskip

We define the {\em critical value} $\ga$ as
\[\ga = \inf \{ \al\in \R \mid \hbox{\eqref{HJa} admits subsolutions}\}\]
The infimum in the definition of $\ga$ is actually a minimum, as
made precise below.

%
\smallskip

\begin{Proposition}  The critical system {\rm(HJ$\gamma$)} is the unique  in the one--parameter family
\eqref{HJa}, $\al \in \R$, for which there are  solutions.
\end{Proposition}

\smallskip


Following  \cite{DaZa}, we give the definition of the Aubry set $\A\subset\T^N$ from
the PDE point of view:
\begin{Definition} A point $y$ belongs to the Aubry set $\A$ if
any maximal critical subsolution  taking a given value at $y$ is a solution to
{\rm(HJ$\gamma$)}.
\end{Definition}
Roughly speaking the Aubry set, which is a closed
nonempty subset of $\T^N$,  is the place where it is concentrated
the obstruction in getting subsolutions of system below the critical
level.  More specifically, there cannot be any critical subsolution
which is, in addition, locally strict at a point in $\A$, in the
sense of the above definition.

\begin{Definition}
For a given critical subsolution $\u$, a component $u_i$, for some
$i \in \unoM$,  is said {\em locally strict} at a point $y \in \T^N$
if there is a neighborhood $U$ of $y$ and a positive constant $\de$
with
\[ H_i(x,Du_i) + \La^i \cdot \u \leq \ga- \de \quad\hbox{a.e. $x \in
\ U$.}\]
\end{Definition}

In analogy with the scalar case, we have a following property:
\smallskip

\begin{Proposition}[{\rm\cite[Proposition 3.9]{DaZa}}]\label{DZ1}  A point
 $y \not \in \A$ if and only if for any given
index $ i \in \unoM$, there exists a critical subsolution $\u$ with
$u_i$ locally strict at $y$.
\end{Proposition}

An interesting fact pointed out in  \cite{DaZa} is that there is a
restriction on the values that a  subsolution to \eqref{HJa} can
attain at a given point. This is a property due to the vectorial
structure of the problem and has no counterpart in the scalar case.
The authors refer to it as  {\em rigidity property} or rigidity
phenomenon. For $\al \geq \ga$, we define for $x \in \T^N$
\begin{equation}\label{defF}
    F_\al(x)= \{ \b \in \R^M \mid \exists \; \u \;\hbox{subsolution to
    \eqref{HJa} with }\u(x)= \b\}.
\end{equation}
Notice that $F_\al(x)$ is convex because of the convex character of
the Hamiltonians, in addition, if $\b \in F_\al(x)$ then $\b + \mu \,
\uno$  is still in $F_\al(x)$ for any $\mu \in \R$, being $\uno \in
\ker(\La)$. This is in a sense equivalent of adding a constant to a
subsolution in the scalar case.
We have a following rigidity phenomenon on $\A$:
\begin{Proposition}[{\rm\cite[Theorem 5.1]{DaZa}}]\label{DZ2}
The  admissible values for  critical
subsolutions  at a given point in $\A$ are of the form
\[ \b + \mu \, \uno\]
where $\b\in\R^M$ depending on $y$,  and $\mu\in\R$.
\end{Proposition}

\section{Random setting} \label{random}
\parskip +3pt
\subsection{A family of probability measures}\label{family}  To build up
the random frame appropriate  for systems, we
introduce a family of  probability measures defined on $\D$, namely
the space of cadlag paths taking values in $\unoM$ endowed with the
$\sigma$--algebra $\F$, see Appendix \ref{due}. Averaging with
respect to such measures will play a crucial role in the subsequent
analysis. We will more precisely show that the  coupling matrix
$\La$ induces a correspondence between the simplex  $\S$ of
probability vectors of $\R^M$, and a simplex of probability measures
on $\D$.

It is convenient for later use to start by recalling that the family
of cylinders of $\F$, or of  $\F_t$ for any $ t \geq 0 $, is a {\em
semi--ring}.
Namely it contains the empty set, is closed by finite
intersections,
 and the difference of two cylinders is a finite disjoint union of cylinders.
Therefore,  taking into account that  $\F$, $\F_t$ are generated by
cylinders, we get  by the Approximation Theorem for Measures, see
\cite[Theorem 1.65]{Kl}.

\smallskip

\begin{Proposition}\label{klenke}
Let $\mu$ be a finite measure on $\F$. For any $E
\in \F$, there is a sequence $E_n$ of multi--cylinders
{\rm(}see {\rm Terminology \ref{cylinder})} in $\F$ with
\[\lim_n \mu(E_n  \triangle E)=0, \]
where $\triangle$ stands for the symmetric difference. If in
addition $E \in \F_t$ for some  $t \geq 0$, then the approximating
multi--cylinders $E_n$ can be taken in $\F_t$.
\end{Proposition}

\medskip

As a consequence we see  that two finite measures on $\D$ coinciding
on the family of cylinders, are actually equal.

We go on, as announced, by performing a converse construction,
namely by defining through the coupling matrix $\La$, for any $\a
\in \S$,  a suitable function on  cylinders and then uniquely
extending it to a probability measure on $\D$.

For a probability vector $\a \in \R^M$, we define  for any
cylinder $\mathcal C(t_1, \cdots,t_k;j_1, \cdots,j_k)$
\begin{equation}\label{newprobaa}
    \mu_\a (\mathcal C(t_1, \cdots,t_k;j_1, \cdots,j_k)) = \left (\a \,e^{-t_1\La} \right )_{j_1} \,
\prod_{l=2}^{k} \left (e^{-(t_l-t_{l-1})\La} \right )_{j_{l-1}
\,j_l}.
\end{equation}

This function enjoys the following key properties:

\begin{itemize}
    \item[{(i)}] it is, for any $k \in \N$, a probability measure on the family of cylinders
    of the form $\mathcal C(t_1, \cdots,t_k;j_1,
    \cdots,j_k)$ obtained by keeping  $((t_1, \cdots,t_k)$  fixed   and varying $(j_1,
    \cdots,j_k)$  in $\unoM^k$, which is actually    a $\sigma$--algebra being
    in a one--to--one correspondence with the family of all subsets of
    $\unoM^k$;
    \item[{(ii)}] if $(t_{i_1}, \cdots, t_{i_l})$ is a subsequence of $(t_1,
    \cdots,t_k)$ with $l < k$ then for any $(j^*_{i_1},
    \cdots,j^*_{i_l}) \in \unoM^l$
    \[ \mu_\a (\mathcal C(t_{i_1}, \cdots, t_{i_l};j^*_{i_1}, \cdots,j^*_{i_l}) =
    \sum_{(j_1, \cdots,j_k) \in J} \mu_\a(\mathcal C(t_1, \cdots,t_k;j_1,
    \cdots,j_k)), \]
    where
    \[J= \{(j_1, \cdots,j_k) \mid j_{i_m}=j^*_{i_m} \;\hbox{for $m=1,
    \cdots, l$}\}.\]
\end{itemize}
The latter condition is known as the Kolmogorov Consistency Condition
and its  validity in this context depends upon $e^{-s\La}$ being a
stochastic matrix for any $s$, which is in turn equivalent, as
showed  in Proposition  \ref{stochazz}, to requiring {\bf (H4)},
{\bf (H5)} on the coupling matrix $\La$.

We are then in position to use the Kolmogorov Extension Theorem, see
for instance \cite[Theorem 14.36]{Kl}, \cite[Theorem 1.2]{SW}, which
ensures, under the previous conditions (i), (ii),  the
existence of an unique probability measure, denoted by $\Pr_\a$, on
$(\mathcal D, \F)$ which extends $\mu_\a$ on the whole $\F$.

\smallskip

It comes from \eqref{newprobaa}  that the map
\[\a \mapsto \Pr_\a  \qquad\hbox{is linear,}\]
consequently  the measures $\Pr_\a$, for $\a=(a_1, \cdots,a_m)$
varying among probability vector of $\R^M$, make up a {\em simplex
of measures} spanned by $\Pr_i:= \Pr_{\ei}$, for $i \in \unoM$, and
\[\Pr_\a= \sum_{i=1}^{M}\ai \, \Pr_i.\]

 Since by \eqref{newprobaa} the measures
$\Pr_i$ are supported in $\mathcal D_i \in \F_0$ (see \eqref{defDi}
for the definition of $\D_i$), we also deduce
\[\Pr_\a(A)= \sum_{i=1}^{M} \ai \, \Pr_i( A \cap \mathcal
D_i) \qquad\hbox{for any $A \in \F$},\]
and
\begin{equation*}\label{newprobac}
a_i= \Pr_\a(\D_i) \qquad\hbox{for any $i \in \unoM$.}
\end{equation*}

Also notice that  all measures  $\Pr_\a$ corresponding to strictly
positive $\a$ are equivalent in the sense that they have the
 same null sets, and these are  the $E \in \F$ with
\[\Pr_i(E)=0 \qquad\hbox{for any $i$.}\]

\bigskip
\begin{Terminology} \label{termini} By a {\em random variable} we mean  any measurable
map from $(\D,\F)$  to a Polish space endowed with the Borel
$\sigma$--algebra.   A  {\em simple random variable} is one that
takes on finitely many values.
We denote by $\E_\a$ the expectation operators relative to $\Pr_\a$, and  put
for simplicity  $\E_i$ in place of $\E_{\ei}$.
We say that some property  holds almost surely, a.s. for short, if
it is valid up to a $\Pr_\a$--null set,  for some, and consequently
for all  $\a >0$, where $>$ must be understood componentwise.

\end{Terminology}

\medskip

We consider the push--forward of the probability measure $\Pr_\a$,
for any $\a \in \S$, through  the flow $\phi_h$ on $\D$ defined in
\eqref{flow}. For a  cylinder $C:=\mathcal C(t_1, \cdots,t_k;j_1,
\cdots,j_k)$, we have for any $\a \in \S$
\begin{eqnarray*}
  \phi_h \# \Pr_\a (C) &=& \Pr_\a \{ \om \mid \phi_h(\om) \in C\}=
\Pr_\a(\mathcal C(t_1+h, \cdots,t_k+h;j_1, \cdots,j_k)) \\
   &=& \left (\a
\,e^{-(t_1+h)\La} \right )_{j_1} \, \prod_{l=2}^{k-1} \left
(e^{-(t_l-t_{l-1})\La } \right )_{j_l \,j_{l-1}}= \Pr_{\a \, e^{-h\La}} (C),
\end{eqnarray*}
which implies
\[ \phi_h \# \Pr_\a (E) = \Pr_{\a \, e^{-h\La
}} (E) \qquad\hbox{for any $E \in \F$}.\]

We have therefore established:
\smallskip

\begin{Proposition}\label{transferflow} For any $\a \in \S$,
$h \geq 0$,
\[ \phi_h \# \Pr_\a = \Pr_{\a \, e^{-h\La}}.\]
\end{Proposition}

\smallskip

Accordingly,  for any measurable function $f: \D \to \R$, we have by
the change of variable formula
\begin{equation}\label{change}
   \E_\a f(\phi_h)  = \int_\D f(\phi_h(\om))\, d \Pr_\a= \int_\D f(\om) \, d \phi_h \#
\Pr_\a =  \E_{\a \, e^{-\La h}} f.
\end{equation}
We consider, for  $t >0$,  the random variables  with values in
$\unoM$ given by the evaluation maps at  $t$, i.e., $\om \mapsto \om(t)$.
By \eqref{newprobaa},
\[\om(t) \# \Pr_\a (i)=  \Pr_\a
(\{\om \mid \om(t) = i\}) = \left (\a \, e^{-t\La} \right )_i\] for
any index $i \in \unoM$, so that
 \begin{equation}\label{newprobaciccio}
\om(t) \# \Pr_\a =  \a \, e^{-t\La}.
\end{equation}
Consequently, if we look at an
$M$--dimensional vector, say $\b$, as a (measurable) function from
$\unoM$ to $\R$, we have
\begin{equation}\label{newprobabb}
   \E_\a b_{\om(t)} = \a \, e^{-t\La} \cdot \b.
\end{equation}

\medskip

Formula \eqref{newprobaciccio}  can be partially recovered  for
measures of the type $\Pr_\a \mres E$ ($\Pr_\a$ restricted to $E$),
where $E$ is any set in $\F$.

\smallskip

 \begin{Lemma}\label{probaproba} For a given $\a \in \S$, $E \in \F_t$ for
 some $t \geq 0$, we have
 \[  \om(s) \# (\Pr_\a \mres E) = \big (  \om(t) \# ( \Pr_\a \mres E) \big )  \,   e^{-(s-t)\La } \quad\hbox{for
    any $s \geq t$}.\]
\end{Lemma}

\begin{proof} We first assume $E$ to be a  cylinder
\[E= \mathcal C(t_1, \cdots, t_k;j_1, \cdots, j_k)\]
for some times and indices. Then the condition $E \in \F_t$ is
equivalent to $ t \geq t_k$.  We have
\[ \om(t_k) \# (\Pr_\a \mres E)(i) = \Pr_\a( E \cap \mathcal C(t_k;i))\]
which implies
\[ \om(t_k) \# (\Pr_\a \mres E) = \Pr_a(E) \, \mathbf e_{j_k}\]
and, according to the definition of $\Pr_\a$ in
\eqref{newprobaa}
\[ \om(s) \# (\Pr_\a \mres E) = \big ( \om(t_k) \# (\Pr_\a \mres
E)\big ) \, e^{-(s-t_k)\La} \quad \hbox{for $s > t_k$.}\]
Consequently,
\[\om(s) \# (\Pr_\a \mres E) = \big ( \om(t_k) \# (\Pr_\a \mres
E)\big ) \, e^{-(t-t_k)\La} \, e^{-(s-t)\La}= \big ( \om(t) \#
(\Pr_\a \mres E)\big ) \, e^{-(s-t)\La}\] for $s \geq t$, as
claimed. The result can be extended by linearity to any
multi--cylinder.

Finally, if $E$ is any set in $\F$, then we consider a sequence of
multi--cylinders $E_n$ in $\F_t$ with $\Pr_\a(E_n \triangle E) \to
0$. By Proposition \ref{klenke},
\[\lim_n \om(s) \# (\Pr_\a \mres E_n) (i)= \lim_n  \Pr_\a( E_n \cap \mathcal C(s;i)) =
\Pr_\a( E \cap \mathcal C(s;i)) = \om(s) \# (\Pr_\a \mres E) (i).
\]
Therefore,
\[
\om(s) \# (\Pr_\a \mres E) =\lim_n  \om(s) \# (\Pr_\a
\mres E_n)  =  \big ( \om(t) \# (\Pr_\a \mres E)\big ) \, e^{-(s-t)\La
}.\qedhere
\]
\end{proof}

\subsection{Stopping times}\label{tre}
 A {\em stopping time}, adapted to $\F_t$, see Appendix \ref{due}, is a nonnegative random variable $\tau$, see
Terminology \ref{termini}, satisfying
\[\{ \tau \leq t\}   \in \mathcal F_t \qquad\hbox{for any $t$,} \]
which also implies $ \{ \tau < t\}, \, \{ \tau = t\}  \in \F_t$.

\medskip

For a bounded random variable $\tau$, we set
  \begin{equation}\label{stoppingbis}
    \tau_n = \sum_j \frac j{2^n} \, \I(\{\tau \in
    [(j-1)/2^n,j/2^n)\}),
\end{equation}
where $\I(\cdot)$ stands for the {\em indicator function} of the set
at the argument, namely the function equal $1$ at any element of the
set and $0$ in the complement. The above sum is finite, being $\tau$
bounded, so the $\tau_n$ are simple stopping times, and letting $n$
go to infinity we get:
\smallskip

\begin{Proposition} \label{stopping} For a bounded stopping time
$\tau$, $\tau_n$  defined as in \eqref{stoppingbis} make up a
sequence of simple stopping times with
\[
\tau_n \geq \tau, \quad
\tau_n \to \tau  \quad\hbox{uniformly in $\D$ as $n\to\infty$.}
\]
\end{Proposition}

\medskip

We consider a  simple  stopping time of the form
\begin{equation}\label{newtau0}
 \tau= \sum_{j=1}^l t_j \,\I(E_j)
\end{equation}
where  the sequence $t_1, \cdots, t_l$ is strictly increasing and
$E_j$ are mutually disjoint sets of $\F$, in addition $E_j \in
\F_{t_j}$ by the very definition of stopping time. The symbol
$\I(\cdot)$ stands again for the indicator function.

We define
\[F_j= \{\tau \geq t_j\},\]
so that
\begin{equation*}\label{newtau1}
   F_j \in \F_{t_{j-1}} \qquad\hbox{for any $j$.}
\end{equation*}
 It  is clear that
\begin{align}
&E_j = \bigcap_{i=1}^j F_i \setminus F_{j+1}, \qquad
F_1 = \D, \nonumber\\
&
F_j = \D \setminus \bigcup_{i=1}^{j-1} E_i \quad \hbox{for $j >1$}, \qquad
F_l=E_l.
\label{newnewtau3}
\end{align}
We derive  that $\tau$ can be  equivalently expressed as
\begin{equation}\label{newtau3}
    \tau= \sum_{j=1}^l (t_j - t_{j-1} ) \, \I(F_j),
\end{equation}
where  we have set  $t_0=0$ to simplify notations. The two
expression of $\tau$ given by \eqref{newtau0},  \eqref{newtau3} are
different: in  \eqref{newtau0} the sets $E_j$ are mutually disjoint
while in \eqref{newtau3} they are decreasing with  respect to $j$.

\bigskip

For a stopping time $\tau$, we consider the map defined as
\begin{equation}\label{brower}
   \a \mapsto  \om(\tau) \# \Pr_\a,
\end{equation}
since the push--forward  of $\Pr_a$ through $\om(\tau)$ is a
probability measure on $\unoM$, which can be identified with an
element of $\S$,  we see that the relation in \eqref{brower} defines
a map from $\S$ to $\S$ which is, in addition,  linear.  Thanks to
Proposition \ref{stocha1}, it can consequently be represented by a
stochastic matrix,  we will denote analogously  to the case of deterministic times, see \eqref{newprobaciccio} ,
 by $e^{- \La \tau}$, acting on the
right. In other terms
\begin{equation}\label{brower1}
  \a \,  e^{- \tau\La}  = \om(\tau) \# \Pr_\a \qquad \hbox{for any $\a
\in \S$.}
\end{equation}

\subsection{Admissible controls}\label{ammetti}

We call  {\em  control}  any random variable $\Xi$  taking values in
$\DiR$ such that

\begin{itemize}
    \item[(i)] is locally (in time) bounded, i.e. for any $t >0$ there is
    $R >0$ with
    \begin{equation}\label{contr1}
    \sup_{[0,t]} |\Xi(t)| < R \qquad\hbox{ a.s.}
\end{equation}
    \item[(ii)] is {\em nonanticipating}, namely for any $t >0$
    \begin{equation}\label{contr2}
    \om_1=\om_2 \;\hbox{in $[0,t]$} \;\; \Rightarrow  \;\; \Xi(\om_1)=\Xi(\om_2) \;\hbox{in
    $[0,t]$.}
\end{equation}
\end{itemize}

Second condition can be equivalently rephrased   requiring $\Xi$ to
be adapted  to the filtration $\F_t$, namely requiring that $\Xi(t)$
is $\F_t$--measurable for any $t$. In fact, if \eqref{contr2} holds
true then the value of $\Xi(\om)(t)$ just depends on the restriction
of $\om$ to $[0,t]$ which actually implies that $\Xi(t)$ is
$\F_t$--measurable. The converse implication comes from  a version
of Doob--Dynkins Lemma for Polish spaces, see \cite{Ka} Lemma 1.13,
asserting that if the $\sigma$--algebra spanned by a random variable
$\# \,1$ is contained in that spanned by  $\# \,2$ then $\# \,1$ is
a measurable function of $\# \,2$  . In our case
 $\# \,1$ is $\Xi(s)$ for   $s \in [0,t]$  and  $\# \,2$ is
 \[ \om \mapsto \;\hbox{restriction of $\om$ to $[0,t]$}\]
which  takes value in $\D(0,t;\unoM)$.

Being the paths in $\DiR$ right continuous, the condition of being
adapted  implies, see \cite{SW} p. 71, that $\Xi$ is in addition
{\em progressively measurable}, namely, for any $t$ the map
\[ (s,\om) \mapsto \Xi(s,\om)\] from  $[0,t] \times \DiM$ to $\R^N$
is measurable with respect to the $\sigma$--algebras $\mathcal B[0,
t] \times \mathcal \F_t$ and $\mathcal B$, where  $\mathcal B[0,
t]$, $\mathcal B$ denote the family of Borel sets of $[0, t]$ and
$\R^N$ with respect to the natural topology.
 We
will denote by $\K$ the class of admissible controls.

\medskip

For a control  $\Xi$, $\mathcal I(\Xi)$ is also a random variable
with values in $\CiT$,  in addition $\mathcal I(\Xi)$ is adapted and
consequently progressively measurable.

\bigskip

For a time $t$, we say that a control   is {\em piecewise
constant} in $[0,t]$ if it is of  the form
\[  \sum_{k=1}^m X_k \, \I([s_k,s_{k+1})) \qquad\hbox{in $ [0,t)$}\]
for some $\F_{s_k}$--measurable $\R^N$--valued bounded random
variables $X_k$, where
\begin{equation}\label{zadar1}
   \hbox{$s_k$ is an increasing finite sequence with $s_1=0$,
   $s_m=t$}
\end{equation}
and $\I(\cdot)$ is as usual the indicator function.  For any
control $\Xi$ and $s_k$ as in \eqref{zadar1}, then the $\Xi(s_k)$
are
  $\F_{s_k}$--measurable $\R^N$--valued bounded random
variables for any $k$, so that
\[\Xi_0 = \left \{ \begin{array}{ll}
            \sum_{k=1}^m \Xi(s_k) \, \I([s_k,s_{k+1})) & \hbox{in $ [0,t)$} \\
            \Xi & \hbox{in $ [t,+ \infty)$} \\
          \end{array} \right .\]
is a control piecewise constant in $[0,t]$.
 We therefore directly derive from Proposition \ref{piece1}:
\smallskip

\begin{Proposition}\label{piece2} For any control $\Xi$ and $t >0$,  there is a sequence of  controls $\Xi_n$ piecewise constant in $[0,t]$  and locally (in time) uniformly bounded with
\[\Xi_n \to \Xi \qquad\hbox{in the Skorohod sense in $\DiR$, for any $\om$.}\]
\end{Proposition}

\bigskip

\section{ An estimate for subsolutions}\label{estimate}

 \parskip +3pt

 For $\al \geq \ga$, an initial point $x$ in
$\T^N$, a bounded stopping time $\tau$ and a control $\Xi$, we
consider in this section the action functional

\begin{equation}\label{action}
    \mathbb E_\a \left [
\int_0^\tau L_{\om(s)}(x + \mathcal I(\Xi)(s),-\Xi(s)) + \al \, ds
\right ].
\end{equation}
\smallskip
Notice that  $I(\Xi)(\tau)$ belongs to $\T^N$ for any $\om$, see
\eqref{proje}. The meaning of the sum between elements of $\T^N$  is
made precise in Notation \ref{notproje}.

\smallskip

We aim at proving:
\smallskip

\begin{Theorem}\label{maxi}  For $\al \geq \ga$, let $\u$, $\tau$, $\Xi$, $\a$
 be a  subsolution to \eqref{HJa},
a bounded stopping time, a control  and
 a probability vector in $\S$, respectively.
 For any initial point $x \in \T^N$,  we have
\begin{equation}\label{max00}
    \E_\a  \big [ u_{\om(0)}(x) - u_{\om(\tau)}(x + \mathcal I(\Xi)(\tau)) \big]  \leq  \mathbb E_\a \left [
\int_0^\tau L_{\om(s)}(x + \mathcal I(\Xi)(s),-\Xi(s)) + \al \, ds
\right ].
\end{equation}
 \end{Theorem}

\medskip

The difficulty in proving Theorem \ref{maxi}  is that the two
integrals appearing in \eqref{max00} do not commute  due to the
presence of the random time $\tau$.
It is worthwhile to point out that this difficulty never happen in the study of evolutionary problem
for weakly coupled systems (see \cite[Proposition 2.5]{MiTr2} for more details).
Joint measurability properties
guarantee that the Fubini theorem can be applied in regions where
stopping time is constant. The idea is then to approximate $\tau$ by
a sequence of simple stopping times $\tau_n$  and then exploit the
subsolution property of $u$ separately  in the regions where
$\tau_n$ are constant. We will take advantage of some properties
about
 probability measures  $\Pr_\a$ we have gathered in
Section \ref{random}.

 \medskip
Throughout the section we put $\al =0$ to ease notations.

\medskip

\begin{Lemma}\label{second}  Let $\u$, $\a$ be as in the statement of Theorem
{\rm\ref{maxi}}, we further consider  $t_2 > t_1 \geq 0$, $E\in
\F_{t_1}$, $\xi_0 \in \DiR$, and  $z_0 \in \T^N$. Then
\begin{eqnarray*}
  && \int_E   \big (u_{\om(t_1)}(z_0)-   u_{\om(t_2)}(z_0 + \mathcal I(\xi_0)(t_2-t_1)) \big )  \, d\Pr_\a\\
   &\leq& \int_E \,  \left ( \int_{t_1}^{t_2} L_{\om(s)}(z_0 +
\mathcal I(\xi_0)(s-t_1),-\xi_0(s)) \, ds \right ) \, d\Pr_\a.
\end{eqnarray*}
\end{Lemma}

\begin{proof}
We can assume $z_0=0$ without loosing generality in the proof.
Since $\u$ is a subsolution to \eqref{HJa}, we have
 \begin{equation}\label{second1}
   - p \cdot q \leq L_i(z,-q) + H_i(z,p)\leq  L_i(z,-q)  - \La^i
\, \u(z)
\end{equation}
 for any $i \in \unoM$, $z \in \T^N$, $q \in \R^N$, $p \in \partial
u_i(z)$ (see Remark \ref{solu}).
We define
\[\d = \om(t_1) \#( \Pr_\a \mres  E),\]
and we have for a.e. $s \in (t_1,t_2)$
\begin{align*}
&\frac d{ds}\left ( \left (\d \, e^{- (s-t_1)\La } \right )
\cdot \u( \mathcal I(\xi_0)(s-t_1)) \right ) \\
&=
\left( \left (\d \, e^{- (s-t_1)\La}
\right )
 \cdot \big ( - \La \, \u( \mathcal I(\xi_0)(s-t_1)) +(p^1(s-t_1) \cdot \xi_0(s-t_1), \cdots, p^M(s-t_1) \cdot \xi_0(s-t_1)) \big) \right),
\end{align*}
where $p^i(s-t_1)$ is a suitable element in $\partial u_i(\mathcal I(\xi_0))(s-t_1)$ for any $i$.
Combining this last  equality with \eqref{second1} and setting $\L=
(L_1, \cdots, L_M)$, we deduce
\[- \frac d{ds}\left ( \left (\d \, e^{- (s-t_1)\La } \right )
\cdot \u( \mathcal I(\xi_0)(s)) \right )  \leq \left ( \d \, e^{-
(s-t_1)\La } \right ) \cdot \L( \mathcal I(\xi_0)(s),-\xi_0(s)), \]
and consequently
\begin{eqnarray*}
&&\d \cdot \u(\mathcal I(\xi_0)(t_1))- \d \cdot  e^{- (t_2-t_1)\La } \u( \mathcal I(\xi_0)(t_2)) =
\int_{t_1}^{t_2} - \frac d{ds}\left ( \left (\d \, e^{- (s-t_1)\La }
\right )
\cdot \u( \mathcal I(\xi_0)(s)) \right ) \, ds \\
&&\leq \int_{t_1}^{t_2} \left ( \d \, e^{- (s-t_1)\La} \right )
\cdot (\L( \mathcal I(\xi_0)(s),-\xi_0(s))  \, ds.
\end{eqnarray*}
We have by the definition of $\d$, \eqref{newprobabb},
Lemma \ref{probaproba}, $E \in \F_{t_1}$
\begin{eqnarray*}
   \int_E \big ( u_{\om(t_1)}(\mathcal I(\xi_0)(t_1)) - u_{\om(t_2)}(\mathcal I(\xi_0)(t_2))\big ) \, d\Pr_\a&=&
    \d \cdot (\u(\mathcal I(\xi_0)(t_1)-  e^{- (t_2-t_1)\La } \u(\mathcal I(\xi_0)(t_2))\\
 \int_E  \, L_{\om(s)}( \mathcal I(\xi_0)(s),-\xi_0(s))  &=& \left ( \d \, e^{- (s-t_1)\La } \right )
\cdot (\L( \mathcal I(\xi_0)(s),-\xi_0(s))
\end{eqnarray*}
 for any $s$ in $[t_1,t_2]$. By
plugging these relations in the last inequality and using the Fubini theorem, we get
\[ \int_E \big ( u_{\om(t_1)}(\mathcal I(\xi_0)(t_1)) - u_{\om(t_2)}(\mathcal I(\xi_0)(t_2))\big ) \, d\Pr_\a
 \leq \int_E \left ( \int_{t_1}^{t_2}
 (\dot{}L_\om( \mathcal I(\xi_0),-\xi_0)  \, ds \right ) d \Pr_\a.\qedhere
\]
\end{proof}

\bigskip

\begin{Lemma}\label{first} For a control $\Xi$ and a bounded stopping time $\tau$,  let $\Xi_n$, $\tau_n$
 be  sequences of  controls  and bounded stopping times,
respectively, with
\begin{eqnarray} \label{first00}
  \Xi_n &\to& \Xi \qquad\hbox{a.s. with respect to Skorohod metric}\\
  \label{first00bis}
  \tau_n &\to& \tau \qquad\hbox{uniformly in $\D$}\\
  \nonumber
  \tau_n & \geq &\tau \qquad\hbox{a.s. for any  $n$.}
\end{eqnarray}
Assume in addition  that  for any   $T>0$, there is  $R= R(T)>0$
with
\begin{equation}\label{first0}
\sup_{s\in[0,T]} |\Xi_n(s)| <  R \quad\hbox{a.s. for any
$n$ }
\end{equation}
Then
\[ \mathbb E_\a
\left [ \int_0^{\tau_n} L_{\om}(x + \mathcal I(\Xi_n),-\Xi_n) \, ds
\right ]
\] converges in $\R$  to
\[ \mathbb E_\a
\left [ \int_0^\tau L_\om(x + \mathcal I(\Xi),-\Xi)  \, ds \right ]
\]
and
\begin{equation}\label{first33}
    \E_\a  \big[ u_{\om(0)}(x) - u_{\om(\tau_n)}(x + \mathcal I(\Xi_n)(\tau_n) \big
  ] \to  \E_\a  \big [ u_{\om(0)}(x) - u_{\om(\tau)}(x + \mathcal I(\Xi)(\tau) \big
  ]
\end{equation}
for any  $x \in \R^N$, $\a \in \S$.
\end{Lemma}

\begin{proof}
 We set $x=0$. We know that conditions \eqref{first00}  \eqref{first0} hold true  outside a $\Pr_\a$--null set
 denoted by  $\D'$. If $\om \in \D \setminus \D'$ the $\Xi_n(\om)$ are uniformly bounded in $[0,\tau(\om)]$, and   we derive from
 \eqref{ciccio1}, \eqref{ciccio21}   that $\Xi_n(\om)$ converges
 pointwise a.e. in $[0,\tau(\om)]$ to $\Xi(\om)$.  Taking  also into account the
 continuity of $L$
and  $\mathcal I(\cdot)$, see   Proposition \ref{contJ}, we get
through the dominated convergence theorem
 \begin{equation}\label{firstgiu}
    \int_0^\tau L_{\om}( \mathcal I(\Xi_n),-\Xi_n) \, ds \;\; \longrightarrow  \;\; \int_0^\tau L_{\om}
( \mathcal I(\Xi),-\Xi) \, ds \quad\hbox{a.s.}
\end{equation}
 Let $T$ be such that $\tau \leq
T$ a.s.,  by \eqref{first0}
\[\max_{s\in[0,T]} |\mathcal I(\Xi_n)(s)| < R \, T \qquad\hbox{for any $n$, outside  a $\Pr_\a$--null set,}\] and
consequently the sequence
\[\int_0^\tau L_{\om}( \mathcal I(\Xi_n),-\Xi_n) \, ds\]
is a.s. uniformly bounded.  Taking also  \eqref{firstgiu} into account,
we can thus obtain the claimed convergence with $\tau$ in place of
$\tau_n$ in the approximating sequence, again via the dominated
convergence theorem. We further have
\begin{eqnarray*}
 \left | \int_0^{\tau_n} L_{\om}( \mathcal I(\Xi_n),-\Xi_n) \, ds
 -\int_0^{\tau} L_{\om}( \mathcal I(\Xi_n),-\Xi_n) \, ds  \right |\leq
 \int_{\tau}^{\tau_n }  |L_{\om}( \mathcal I(\Xi_n),-\Xi_n)| \, ds
\end{eqnarray*}
Owing to \eqref{first00bis}  and  the uniformly boundedness property of
the integrand, the right hand--side of the above formula becomes
infinitesimal, as $n$ goes to infinity, uniformly in $\om$ so that
\[ \E_\a  \; \left[  \left | \int_0^{\tau_n} L_{\om}( \mathcal I(\Xi_n),-\Xi_n) \, ds
    -\int_0^{\tau} L_{\om}( \mathcal I(\Xi_n),-\Xi_n) \, ds  \right | \; \right
    ] \to 0.\]
This shows the  first  convergence in the statement. Limit relation
\eqref{first33} can be proved similarly using the continuity of $u$
in $\T^N$.
\end{proof}

\medskip

\begin{Lemma}\label{prethird} Assume
\begin{equation}\label{newstop}
    \tau= \sum_{j=1}^l t_j \, \I(E_j)
\end{equation}
to be a simple stopping time, with the $t_j$ making up an increasing
sequence of times,  and set $F_j= \{\tau \geq t_j\}$ for any $j \in
\{1, \cdots, l\}$. Let $\u$, $\Xi$, $\a$, $x$ be as in the statement
of Theorem \ref{maxi}, then
\begin{eqnarray*}
&&\E_\a\left[\int_0^\tau L_\om ( x +\mathcal I(\Xi), - \Xi) + \al \, ds \right]
=
\sum_{j=1}^l
\int_{F_j} \left ( \int_{t_{j-1}}^{t_j}  L_\om ( x + \mathcal I
(\Xi), - \Xi)  + \al \,
ds \right ) \, d\Pr_\a , \\
&& \E_\a \big[u_{\om(0)}(x) - u_{\om(\tau)} (x + \mathcal I(\Xi(\tau))) \big] \\
&&\qquad\qquad\qquad\qquad\qquad=
\sum_{j=1}^l
\int_{F_j} \big ( u_{\om(t_{j-1})}( x + \mathcal I(\Xi(t_{j-1}))) -
u_{\om(t_j)}(x + \mathcal I(\Xi(t_j))) \big ) \, d \Pr_\a  .
\end{eqnarray*}
\end{Lemma}

\begin{proof} We set $t_0=0$ and
\[I = \E_\a \, \left[\int_0^\tau L_\om ( x +\mathcal I(\Xi), - \Xi) + \al \,
ds\right].\] Taking into account the definition of $\tau$ and that the
$t_i$ are monotone, we have
\begin{eqnarray*}
  I &=& \sum_{i=1}^l \int_{E_i} \int_{0}^{t_i} L_\om ( x +\mathcal I(\Xi), - \Xi) + \al \,
ds  = \sum_{i=1}^l \sum_{j=1}^i \int_{E_i} \int_{t_{j-1}}^{t_j}
L_\om ( x +\mathcal I(\Xi), - \Xi) + \al \, ds\\
&=& \sum_{j=1}^l \sum_{i \geq j} \int_{E_i} \int_{t_{j-1}}^{t_j}
L_\om ( x +\mathcal I(\Xi), - \Xi) + \al \, ds
\end{eqnarray*}
and, owing to \eqref{newnewtau3}
\[ \sum_{i \geq j} \int_{E_i} \int_{t_{j-1}}^{t_j} L_\om ( x +\mathcal I(\Xi), - \Xi) + \al \,
ds  = \int_{F_j} \int_{t_{j-1}}^{t_j} L_\om ( x +\mathcal I(\Xi), -
\Xi) + \al \, ds\] for any $j \in \{1, \cdots, l\}$. Therefore,
summing over $j$ we get
\[ I= \sum_{j=1}^l
\int_{F_j} \left ( \int_{t_{j-1}}^{t_j}  L_\om ( x + \mathcal I
(\Xi), - \Xi)  + \al \, ds \right ) \, d\Pr_\a\] as desired. The
second equality in the statement can be proved along the same lines,
we provide some detail for readers' convenience.  We start defining
\[ J= \E_\a \big[ u_{\om(0)}(x) - u_{\om(\tau)} x + \mathcal I(\Xi(\tau))) \big
], \] then we have
\begin{eqnarray*}
  J &=& \sum_{i=1}^l
\int_{E_i} \big ( u_{\om(0)}( x ) - u_{\om(t_j)}(x + \mathcal
I(\Xi(t_j))) \big ) \, d \Pr_\a \\ &=& \sum_{i=1}^l \sum_{j=1}^i
\int_{E_i} \big ( u_{\om(t_{j-1})}( x  + \mathcal I(\Xi(t_{j-1}) -
u_{\om(t_j)}(x + \mathcal
I(\Xi(t_j))) \big ) \, d \Pr_\a\\
&=& \sum_{j=1}^l \sum_{i \geq j} \int_{E_i} \big ( u_{\om(t_{j-1})}(
x  + \mathcal I(\Xi(t_{j-1}) - u_{\om(t_j)}(x + \mathcal
I(\Xi(t_j))) \big ) \, d \Pr_\a
\end{eqnarray*}
and, again exploiting \eqref{newnewtau3}
\begin{eqnarray*}
   && \sum_{i \geq j} \int_{E_i} \big ( u_{\om(t_{j-1})}(
x  + \mathcal I(\Xi(t_{j-1}) - u_{\om(t_j)}(x + \mathcal
I(\Xi(t_j))) \big ) \, d \Pr_\a \\
   &=&
    \int_{F_j} \big ( u_{\om(t_{j-1})}(
x  + \mathcal I(\Xi(t_{j-1}) - u_{\om(t_j)}(x + \mathcal
I(\Xi(t_j))) \big ) \, d \Pr_\a
\end{eqnarray*}
 for any $j \in \{1, \cdots, l\}$.
We conclude the proof summing over $j$.

\end{proof}

\medskip

\begin{Proposition}\label{third} The assertion of Theorem {\rm\ref{maxi}} is true if
we take the stopping time $\tau$ simple, say of the form
\eqref{newstop},  and the control $\Xi$ piecewise constant in $[0,
T]$ for some $T \geq t_l$.
\end{Proposition}
\begin{proof}  Since $T \geq t_l$, we can  assume that   $\Xi$ has
the form
 \[\Xi= \sum_{k=1}^m X_k \, \I([s_{k-1},s_{k})) \qquad \hbox{in $[0,t_l)$,}\]
 where  $X_k$ are $\R^N$--valued random variables and $s_k$ is a finite increasing sequence with
  $s_0=0$ and
 $s_m=t_l$; we can assume  in addition that all the times $t_j$, $j=1, \cdots ,l$
 belong to the sequence.  Consequently,   it can be univocally  associated to
 any interval $[s_{k-1},s_{k})$ an index  $j$ such that
  $[s_{k-1},s_k) \subset [t_{j-1},t_{j})$. Due to the nonanticipating character of $\Xi$
 \begin{equation}\label{third1}
  \hbox{$X_k$ is $\F_{s_{k-1}}$--measurable.}
\end{equation}

We fix   indices $k$, $j$  such that  $[s_{k-1},s_{k})$ is contained
in $[t_{j-1},t_{j})$,  by \eqref{third1}  there is a sequence of
simple $\F_{s_{k-1}}$--random
 variables
 \[ Y_n= \sum_r y^n_r \, \I(B^n_r)\]
 taking values in $\R^N$ and converging a.s. to $X_k$, see \cite[Theorem 1.4.4]{Ku}, with
 $y^n_r \in \R^N$ and $B^n_r \in \F_{s_{k-1}}$ for any $n$.
 Then, slightly modifying the argument in  Lemma \ref{first}, we get that
 \begin{equation}\label{third2}
    \int_{F_j} \int_{s_{k-1}}^{s_{k}} L_{\om(s)}(\mathcal I(\Xi(s_{k-1}))
    + \mathcal I( Y_n)(s-s_{k-1}),- Y_n) \, ds
\end{equation}
converges to
\[\int_{F_j} \int_{s_{k-1}}^{s_{k}} L_{\om(s)}
(\mathcal I(\Xi(s_{k-1})) + \mathcal I( X_k)(s-s_{k-1}),-  X_k) \,
ds\]
 as $n$ goes to infinity, and similarly
 \begin{equation}\label{third3}
\int_{F_j}   \big (u_{\om(s_{k-1})}(\mathcal I(\Xi(s_{k-1}))-
u_{\om(s_{k})}(\mathcal I(\Xi(s_{k-1}) +
 \mathcal I(Y_n)(s_{k}-s_{k-1}) ))) \big )
\end{equation}
converges to
\[\int_{F_j}   \big (u_{\om(s_{k-1})}(\mathcal I(\Xi(s_{k-1}))-
u_{\om(s_{k})}(\mathcal I(\Xi(s_{k-1}) +
 \mathcal I(X_k)(s_{k}-s_{k-1})))) \big )\]

  Due to the form of $Y_k$, the integral
 in \eqref{third2}, \eqref{third3}  can be in turn written as
\begin{align*}
&\sum_r  \int_{F_j \cap B^r_n} \int_{s_{k-1}}^{s_{k}} L_{\om(s)}(\mathcal I(\Xi(s_{k-1})) +  y^n_r \, (s-s_{k-1}),-y^n_r) \, ds \, d \Pr_\a,\\
&\sum_r \int_{F_j \cap B^r_n}   \big\{u_{\om(s_{k-1})}\big(\mathcal I(\Xi(s_{k-1}))-
u_{\om(s_{k})}(\mathcal I(\Xi(s_{k-1}) +
 y^n_r \, (s_{k}- s_{k-1})))\big) \big\}  \, d\Pr_\a,
\end{align*}
respectively.
Since $F_j \in \F_{t_{j-1}}$, $B^r_n \in \F_{s_{k-1}}$ and $s_{k-1}
\geq t_{j-1}$, we deduce  $F_j \cap B^r_n \in \F_{s_{k-1}}$, and we
can apply Lemma \ref{second} to any term of the previous sum. This
yields
\begin{align*}
&\int_{F_j \cap B^r_n}   \big (u_{\om(s_{k-1})}(\mathcal
I(\Xi(s_{k-1}))- u_{\om(s_{k})}(\mathcal I(\Xi(s_{k-1}) +
y^n_r \, (s_{k}- s_{k-1})))) \big )  \, d\Pr_\a \\
\leq &
\int_{F_j \cap B^r_n} \int_{s_{k-1}}^{s_{k}} L_{\om(s)}(\mathcal I(\Xi(s_{k-1})) +  y^n_r \,
   (s-s_{k-1}),
   - y^n_r) \, ds \, d \Pr_\a
\end{align*}
for any $r$.
By summing over $r$ and passing to the limit as $n$
goes to infinity, we further get
\begin{align*}
&\int_{F_j}   \big (u_{\om(s_{k-1})}(\mathcal I(\Xi(s_{k-1}))-
u_{\om(s_{k})}(\mathcal I(\Xi(s_{k-1}) +
 \mathcal I(X_k)(s_{k}- s_{k-1})))) \big )  \, d\Pr_\a  \\
\leq &
\int_{F_j } \int_{s_{k-1}}^{s_{k}} L_{\om(s)}(\mathcal I(\Xi(s_{k-1})) + \mathcal I(X_{k-1})(s- s_{k-1}),-
   X_k) \, ds \, d \Pr_\a.\end{align*}
By summing all inequalities as above corresponding to intervals
$[s_{k-1},s_{k})$ in $[t_{j-1},t_{j})$ we obtain
\begin{align*}
&\int_{F_j}   \big (u_{\om(t_{j-1})}(\mathcal I(\Xi(t_{j-1}))-
u_{\om(t_{j})}(\mathcal I(\Xi(t_{j})))) \big )  \, d\Pr_\a
\leq
\int_{F_j } \int_{t_{j-1}}^{t_{j}} L_{\om(s)}(\mathcal I(\Xi(s)) ,-\Xi(s)) \, ds
\, d \Pr_\a.
\end{align*}

We conclude the proof summing over $j$ and exploiting Lemma
\ref{prethird}.
\end{proof}

\begin{proof}[{\bf Proof of the Theorem \ref{maxi}}] \;
By Proposition \ref{stopping}  $\tau$ can be approximated uniformly
in $\om$ by a sequence of simple stopping times $\tau_n$ with
$\tau_n \geq \tau$ and $\tau_n \leq T$ for some constant $T$, in
addition by Proposition \ref{piece2} $\Xi$ can be approximated a.s.
with respect to Skorohod metric    by a sequence of control $\Xi_n$
piecewise constant  in $[0,T]$ and and locally (in time)
uniformly bounded.

Owing to Proposition  \ref{third},  inequality  \eqref{max00} holds
true if we replace $\tau$, $\Xi$ by $\tau_n$, $\Xi_n$, respectively,
for any $n$. We conclude by  passing at the limit as $n$ goes to
infinity and exploiting Lemma \ref{first}.
\end{proof}

\medskip

\smallskip

\begin{Notation}\label{kappa}
For a bounded stopping time $\tau$ and  a pair $x$, $y$ of elements
of $\T^N$, we set
  \[\K(\tau,y-x)=
  \left \{ \Xi \in \K \mid \mathcal I(\Xi)(\tau)  = y-x \;\hbox{a.s.} \right \}, \]
 notice that both $I(\Xi)(\tau)$ and $y-x$ are elements of $\T^N$, see
 \eqref{proje} and refer to Notation \ref{notproje} for the meaning of $y -x$.
  Also notice that $\mathcal I(\Xi)(\tau)$ is a random variable taking
value in $\R^N$ because $\Xi$ is progressively measurable and $\tau$
is a stopping time. We recall that the  diction a.s.  must be
understood   with respect to the family of equivalent measures
$\Pr_\a$, $\a > 0$.

\noindent We will call, with some abuse of language, the controls
$\Xi$ belonging to $\K(\tau,0)$   {\em $\tau$--cycles}.
\end{Notation}

\medskip

\begin{Remark} For $x$, $y$ in $\T^N$, the family of controls
$\K(\tau,y-x)$ is nonempty whenever $\essinf \tau> 0$. In fact for
such a stopping time select $\eps >0$ with $\eps < \essinf \tau$ and
define a control $\Xi$ setting for any $\om$
\[\Xi(\om)(s)= \left \{ \begin{array}{ll}
                         z_0 & \hbox{for $s \in [0,\eps)$} \\
                          0 & \hbox{for $s \in [\eps, + \infty)$} \\
                        \end{array} \right . \]
where $z_0$ is any vector of $\R^N$ with $\proj( \eps \, z_0) = y-x$
($\proj$ is the projection of $\R^N$ onto $\T^N$). It is indeed
apparent that $\Xi$ belongs to $\K(\tau,y-x)$.

\end{Remark}

\medskip

\smallskip

Using Notation \ref{kappa}, we derive from Theorem \ref{maxi}:
\smallskip

\begin{Corollary}\label{cormaxi}  For any pair of points $x$, $y$ in $\R^N$,  subsolution $\u$ to
\eqref{HJa}, $\a \in \S$, bounded stopping time $\tau$  and $\Xi \in \K(\tau,y-x)$, we have
\begin{equation}\label{formusub}
   \E_\a \big [u_{\om(0)}(x) - u_{\om(\tau)}(y) \big ]
   \leq  \E_\a \left [
\int_0^\tau L_{\om(s)}(x + \mathcal I(\Xi)(s),-\Xi(s)) + \al  \, ds
\right]. \end{equation}
\end{Corollary}

\medskip

In the next section we will show, see Theorem \ref{susuvai}, that
\eqref{max00} actually characterizes subsolutions to \eqref{HJa}.

\smallskip

\bigskip

\section{A representation formula for subsolutions}\label{subsol}
Throughout the section  we consider a constant $\al$ greater than or
equal to $\gamma$.
For  $y$ in $\R^N$,   $\b \in \R^M$, we define
 \begin{equation}\label{formula}
    v_i(x)=   \inf \, \mathbb E_i
 \left [ \int_0^\tau L_{\om}( (x +\mathcal I(\Xi),-\Xi) + \al \, ds +
 b_{\om(\tau)} \right ]
\end{equation}
for any $i \in \unoM$, $x \in \R^N$, where the infimum is taken with
respect to any bounded stopping times $\tau$ and $ \Xi
\in\K(\tau,y-x)$.  We have

\medskip

\begin{Proposition}\label{limita} The function $\v=(v_1, \cdots, v_M)$ defined in \eqref{formula} is
bounded in $\T^N$.
\end{Proposition}
\begin{proof} Taking into account that $\uno \in \ker(\La)$, we see
that if $b_0 \in F_\al(y)$ (see \eqref{defF} for the definition of
$F_\al$) then $b_0 + \mu \, \uno \in  F_\al(y)$ as well,  for any
$\mu \in \R$. We can  consequently find a subsolution  $\u$ to
\eqref{HJa} with
\[ \u(y) \leq \b.\]
Owing to Corollary \ref{cormaxi}, we then have
\begin{multline*}
\mathbb E_i \left [
\int_0^\tau L_{\om(s)}(x + \mathcal I(\Xi)(s),-\Xi(s)) + \al \, ds +
b_{\om(\tau)} \right ] \\
\geq \mathbb E_i \left [
\int_0^\tau L_{\om(s)}(x + \mathcal I(\Xi)(s),-\Xi(s)) + \al \, ds +
 u_{\om(\tau)}(y)\right ] \geq u_i(x).
\end{multline*}
for any $i \in \unoM$, $x \in \R^N$,  bounded stopping time $\tau$
and $ \Xi \in\K(\tau,y-x)$. This implies
\[ \v(x) \geq \u(x) \qquad\hbox{for any $x$},\]
where $\geq$ must be understood componentwise. On the other side, by
setting $\tau \equiv |x-y|$, $\Xi=\frac{x-y}{|x-y|}$ and taking into
account that $\L$ is locally bounded, we see that $\v$ is also
bounded from above.
\end{proof}

\medskip

We aim at showing:

\begin{Theorem}\label{mainsub} The function $\v$ defined by  \eqref{formula}
is subsolution  to \eqref{HJa}.
 \end{Theorem}

 \medskip

We postpone the proof after some preliminary material. The  crucial
point is to prove  a Dynamical Programming Principle type result.
We will use the flow $\phi_h$ defined \eqref{flow} in Appendix and the change of variable formula
\eqref{change}.

\medskip

\begin{Proposition} Let  $h$, $x$, $\xi_0$, $j$  be a positive time, a point in $\R^N$, a path in $\DiR$,
 and an index in $\unoM$ respectively.
 Then
 \begin{equation}\label{subbo1}
   v_j(x) \leq \E_j \left [ \int_0^h  L_{\om}(x+ \mathcal I(\xi_0),-\xi_0) +
   \al \, ds + v_{\om(h)}(x + \mathcal I(\xi_0)(h))  \right].
\end{equation}
\end{Proposition}

\begin{proof} Fix $\eps >0$ and set $\al=0$, $z= x + \mathcal I(\xi_0)(h)$  to ease notation.
Denote, for any $i$, by $\tau^i$, $\Xi_i$ bounded stopping times and
controls in $\K(\tau^i,y-z)$ with
\begin{equation}\label{subbo10}
    v_i(z) \geq  \E_i
 \left [ \int_0^{\tau^i} L_{\om}(z +\mathcal I(\Xi_i),-\Xi_i)  \, ds +
 b_{\om(\tau^i)} \right] - \eps
\end{equation}
 We define  new  stopping times and controls   via
 \[\tau= \tau^i, \quad \Xi=\Xi_i \qquad\hbox{in $\D_i$ for any $i$},\]
 it is clear that $\Xi \in \K(\tau, y-z)$.
We   set
 \[\widetilde \tau (\om)= \tau(\phi_h(\om)) + h \qquad\hbox{for any $\om
 \in \D$,}\]
this is yet a stopping time, since for any $t \geq h$
\[\{\om \mid \widetilde \tau (\om) \leq t\} = \{\om \mid  \tau(\phi_h(\om)) \leq t -h\}
= \phi_h^{-1} \big ( \{\om \mid \tau(\om) \leq t -h\} \big ),\]
which actually yields by Proposition \ref{supershift}
\[ \{\om \mid \widetilde \tau (\om) \leq t\} \in \F_t,\]
as desired. We further set
 \[\widetilde \Xi(\om)(s) = \left \{
\begin{array}{ll}
    \xi_0(s), & \hbox{for $\om \in \D$, $s \in [0,h)$} \\
    \Xi(\phi_h(\om))(s-h), & \hbox{for $\om \in \D$, $s \in [h,+\infty)$.} \\
\end{array}
\right.\]
To justify $\widetilde \Xi$ being an admissible control,
we define a map $\Psi$ from $\DiR$ to itself through
\[ \Psi(\xi)(s)= \left\{
\begin{array}{ll}
    \xi_0(s) & \hbox{for $s \in [0,h)$} \\
    \xi(s-h) & \hbox{for $s \in [h, + \infty)$}.
\end{array}
\right.\]
According to the very definition of convergence in $\DiR$,
this mapping  is continuous in the sense of Skorohod,  in fact if
$\xi_n \to \xi$ and $g_n$ is the corresponding time scale
deformation, then we define
\[\overline g_n(s)= \left \{ \begin{array}{ll}
                             s & \hbox{for $s \in [0,h)$}  \\
                             g_n(s-h)+h & \hbox{for $s \in [h, + \infty)$} \\
                           \end{array} \right . \]
and it is straightforward to check that $\overline g_n$ locally
uniformly converges to the identity function in $[0,+\infty)$ and
$\Psi(\xi)(\overline g_n(s))$ locally uniformly converges to
$\Psi(\xi)(s)$. We can rephrase the definition of $\widetilde\Xi$
above as
\[ \widetilde \Xi (\om)=  \Psi(\Xi (\phi_h(\om)),\]
which shows that $\Xi$ is a random variable as composition of
continuous and    measurable maps. If $\om_1=\om_2$ in $[0,t]$, for
some $t> h$, then
\[ \phi_h(\om_1) = \phi_h(\om_2) \qquad\hbox{in $[0,t-h]$}\]
which implies
\[\Xi(\phi_h(\om_1))=\Xi(\phi_h(\om_2)) \qquad\hbox{in $[0,t-h]$},\]
therefore
\begin{align*}
&\widetilde \Xi(\om_1)= \xi_0= \widetilde \Xi(\om_2)  \qquad\hbox{in $[0,h]$}, \\
&\widetilde \Xi(\om_1(s))= \Xi(\phi_h(\om_1))(s-h)= \Xi (\phi_h(\om_2))(s-h)=
  \widetilde \Xi(\om_2)(s) \qquad\hbox{in $[h,t]$},
\end{align*}
which shows that $\Xi$ is nonanticipating. Finally the
the uniformly boundedness condition is clearly fulfilled. We conclude that
$\widetilde \Xi$ is an admissible control.  To show that it  belongs
to $\K(\widetilde\tau,y-x)$, we consider for $\om \in \D$
\[
\int_0^{ \widetilde \tau(\om)} \widetilde \Xi(\om) \, ds
  = \int_0^h  \xi_0 \,
ds + \int_{h}^{ \widetilde \tau (\om)}  \Xi( \phi_h(\om))(s-h) \, ds
= z - x + \int_0^{  \tau (\phi_h(\om))}  \Xi(\phi_h(\om))(s) \, ds,
\]
Owing to $\Xi \in \K(\tau, y-z)$ and Proposition \ref{transferflow}
we have for any $\a
> 0$ in $\S$
\[ \Pr_\a  \left \{ \om \mid \int_0^{\tau (\phi_h(\om))}  \Xi(
\phi_h(\om))(s) \, ds \neq y-z \right \} = \Pr_{\a e^{-h\La}} \left
\{ \om \mid \int_0^{\tau (\om)}  \Xi( \om)(s) \, ds \neq y-z \right
\} = 0. \] This establishes  that $\widetilde \Xi \in \K(
\widetilde\tau, y-x)$. We compute for $s >0$:
\begin{eqnarray} \label{subbo3}
 x +  \mathcal I(\widetilde \Xi)(\om)(s+h) &=&  x + \int_0^h \xi_0 \, dr +
   \int_h^{s+h} \widetilde \Xi(\om) \, dr \\
 \nonumber  &=& z + \int_0^s \Xi(\phi_h(\om)) \, dr = z + \mathcal
   I(\phi_h(\om))(s).
\end{eqnarray}
According to the very definition of $\v$, we then have
\begin{eqnarray}
 \label{subbo2} \qquad \qquad v_j(x) &\leq&
 \E_j
 \left [ \int_0^{\widetilde\tau } L_{\om}(x +\mathcal I(\widetilde \Xi),-\widetilde \Xi)  \, ds +
 b_{\om(\widetilde \tau)} \right ] \\ \nonumber
   &=& \E_j
 \left [  \int_0^h  L_{\om}(x+ \mathcal I(\xi_0),-\xi_0)
    \, ds + \int_h^{\widetilde\tau } L_{\om}(x +  \mathcal I(\widetilde \Xi),-\widetilde \Xi)  \, ds +
 b_{\om(\widetilde\tau)} \right ].
\end{eqnarray}

Using the definitions of $\widetilde \tau$, $\widetilde \Xi$, the
change of variable formula \eqref{change} and \eqref{subbo3}, we
have
\begin{align*}
& \E_j \left [ \int_h^{\widetilde\tau } L_{\om}(x +
\mathcal I(\widetilde \Xi),- \widetilde  \Xi)  \, ds + b_{\om(\widetilde\tau(\om))} \right ] \\
=&
\E_j \left [ \int_0^{\widetilde\tau - h} L_{\om(s+h)}(x +\mathcal
I(\widetilde \Xi)(\om)(s+h),-\widetilde \Xi(\om)(s+h))  \, ds +
b_{\om(\widetilde\tau(\om))} \right ]\\
 =&
\E_j \left [
\int_0^{\tau(\phi_h(\om))} L_{\phi_h(\om)}(z +\mathcal
I(\Xi)(\phi_h(\om)(s),-\Xi(\phi_h(\om)(s))  \, ds +
b_{\phi_h(\om)(\tau(\phi_h(\om)))\emph{}} \right ]\\
 = &
\E_{\mathbf e_j \, e^{-h \La}} \left [
\int_0^{\tau(\om)} L_{\om}(z +\mathcal I(\Xi)(\om)(s),-\Xi(\om)(s))
\, ds +  b_{\om(\tau)}  \right ]
\end{align*}
Using \eqref{subbo10}, we further get
\begin{align*}
&  \E_{\mathbf e_j \, e^{-h \La}} \left [ \int_0^{\tau(\om)}
L_{\om}(z +\mathcal I(\Xi)(\om)(s),-\Xi(\om)(s)) \, ds  +
 b_{\om(\tau)}\right ] \\
= &
\sum_i \big ( \mathbf e_j e^{-h\La} \cdot \ei \big )
\E_i \left [  \int_0^{\tau ^i} L_{\om}(z +\mathcal
I(\Xi_i)(s),-\Xi_i(s)) \,
ds + b_{\om(\tau^i)}\right ]   \\
 \leq&
 \sum_i \big ( \mathbf e_j \, e^{-h\La} \cdot \ei \big ) \, (v_i(z)+\eps)
  = \mathbf e_j \, e^{-h\La} \cdot \v(z) +
 \eps =  \E_j v_{\om(h)}(z) + \eps.
\end{align*}

Combining the last two computations we get
\[   \E_j \left [ \int_h^{\widetilde\tau } L_{\om}(x +
\mathcal I(\widetilde \Xi),- \widetilde  \Xi)  \, ds +
b_{\om(\widetilde\tau(\om))} \right ]                \leq     \E_j
v_{\om(h)}(z) + \eps \] and recalling \eqref{subbo2} and the
definition of $z$ we finally obtain
\[v_j(x) \leq \E_j \left [ \int_0^h  L_{\om}(x+ \mathcal I(\xi_0),-\xi_0)
 \, ds + v_{\om(h)}(x + \mathcal I(\xi_0)(h))  \right] + \eps.\]
Taking into account that $\eps$ is arbitrary and that we have set
$\al=0$, we obtain in the end the assertion.
\end{proof}

\bigskip

\begin{Lemma}\label{subbolip} The function $\v$ defined by
\eqref{formula} is
 Lipschitz--continuous in $\T^N$.
\end{Lemma}
\begin{proof}
 We consider two points $z \neq x$, and set
$\tau_0 \equiv |z-x|$, $\Xi_0=\frac{z-x}{|z-x|}=:q$. Then, according
to \eqref{subbo1}
\[v_i(x) - \ei \, e^{-|x-z|\La} \cdot \v(z) \leq \E_i
 \left [ \int_0^{|x-z|} L_{\om(s)}  ( x + s \,q, -q )  + \al  \,
 ds \right ]\]
 from which we derive
 \begin{equation}\label{subbolip1}
    v_i(x)-v_i(z) +   \ei \, \left ( I - e^{-|x-z|\La}  \right )  \cdot  \v(z)
 \leq \E_i \left [ \int_0^{|x-z|} L_{\om(s)}  ( x + s \,q, -q )  + \al  \,
 ds \right ]
\end{equation}
 We take  a constant $R$ which is at the same time   upper bound of both $\L(x,q)$ in $\T^N \times
 B(0,1)$ and $|\v(x)|$ in $\T^N$, see Lemma \ref{limita}, and in addition Lipschitz constant of  \[t
 \mapsto \ei \,   e^{-t\La}    \quad\hbox{in $[0,+ \infty)$}\]
 for any  $i$, see Proposition \ref{liplip}. We deduce from \eqref{subbolip1}
 \[v_i(x)-v_i(z) \leq ( R + \al + R^2) \, |x-z|.\]
 This completes the proof.
\end{proof}

\medskip

\begin{proof} [{\bf Proof of Theorem \ref{mainsub}}]
We consider   a point  $x \in \R^N$ where all components of
$\v(x)$ are differentiable,  and fix  a nonvanishing  vector $q \in
\R^N$, further we take  $\xi_0 \equiv q$, and accordingly
\[x +\mathcal I(\xi_0)(s)= x + s \, q \qquad\hbox{for any $s \geq 0$.}\]
Formula \eqref{subbo1} then  reads
\[
v_j(x) -  \mathbf e_j  \, e^{-h\La} \cdot  \v(x + h \, q) \leq
 \int_0^h  \mathbf e_j \,  e^{-s\La} \cdot \L(x+ s \, q,-q) + \al
  \, ds,
\]
which implies
\[ \frac {v_j(x) - \mathbf e_j  \, e^{-h\La} \cdot  \v(x + h \, q)}h \leq
\frac 1h \,  \int_0^h  \mathbf e_j \,  e^{-s\La} \cdot \L(x+ s \,
q,-q) + \al \, ds.\]
  Passing to the limit as $h$ goes to $0$, and taking into account that all the $v_j$
   are differentiable at $x$, we get
\[ \La^j \cdot \v(x) - D v_j(x) \cdot q \leq L_j(x,-q) + \al.\]
Being $q$ arbitrary, we further obtain
\[\La^j \cdot  \v(x) + H_j(x,Dv_j(x)) = \La^j \, \v(x)  + \sup_q \{-  D v_j(x) \cdot q -  L_j(x,-
q)\} \leq \al.\] This shows that $\v(x)$ is a.e.  and so viscosity
subsolution of the system \eqref{HJa}.
\end{proof}

\begin{Theorem}\label{suvai} For $y \in \T^N$,   $\b \in F_\al(y)$ if and only if
\begin{equation}\label{suvai1}
  \mathbb E_i \left  [\int_0^\tau L_{\om(s)}(y + \mathcal I(\Xi)(s),-\Xi(s)) + \al \, ds - b_i
+ b_{\om(\tau)} \right ] \geq 0
\end{equation}
for any $i \in \unoM$, bounded  stopping times $\tau$ and $\tau$--cycles $\Xi$.
\end{Theorem}
\begin{proof}  We denote as usual by $\v$ the function defined in \eqref{formula}. By taking the  stopping time $\tau \equiv 0$ and the control
$\Xi \equiv 0$, we see that
\[ \v(y) \leq \b,\]
where $\leq$ must be understood componentwise. If \eqref{suvai1}
holds then we also get the converse inequality so that $\v(y)=\b$,
which proves $\b \in F_\al(y)$ being $\v$ subsolution to
\eqref{HJa}.

Conversely,  if there is a subsolution $\u$ of \eqref{HJa} with $u(y)=\b$
then \eqref{suvai1} is a direct consequence of Corollary \ref{cormaxi}
\end{proof}

\medskip

We give a characterization of the Aubry set from the Lagrangian point of view.
\begin{Theorem}\label{vai} Assume the element $\b$ appearing in \eqref{formula} to be in  $ F_\al(y)$, then
\begin{itemize}
    \item [{\rm (i)}] $\v(y)= \b${\rm;}
    \item [{\rm (ii)}] $\v$ is the maximal subsolution to \eqref{HJa}
    taking the value $\b$ at $y${\rm;}
    \item [{\rm (iii)}] If $\al=\ga$ and $y \in \A$ then $\v$ is a
    critical solution.
\end{itemize}
\end{Theorem}

\begin{proof}
Item (i) has been already proved in Theorem  \ref{suvai}.  If $\u$ is a subsolution to \eqref{HJa}
with $\u(y)= \b$, then   by Corollary \ref{cormaxi} we get
\[u_i(y) \leq \mathbb E_i \left [
\int_0^\tau L_{\om(s)}(x + \mathcal I(\Xi)(s),-\Xi(s)) + \al \, ds +
 b_{\om(\tau)} \right ]\]
 for any $i \in \unoM$, bounded stopping time $\tau$ and
 $\tau$--cycle $\Xi$. This shows
  \[ \v \geq \u.\]
 Item (iii)
directly comes from the definition of the Aubry set.
\end{proof}

\bigskip

We finish the section by showing that  for any $\al \geq \ga$
inequality \eqref{formusub} actually characterizes subsolutions to
\eqref{HJa}.

\smallskip

\begin{Theorem}\label{susuvai}  A function $\u:\T^n\to\R^M$
is a subsolution to
\eqref{HJa} if and only if  inequality \eqref{formusub}  holds true
for any pair of points $x$, $y$  in $\T^N$, $\a \in \S$, any bounded
stopping time $\tau$,   $\Xi \in \K(\tau, y-x)$.
\end{Theorem}

\medskip

In view of Corollary \ref{cormaxi} , it is enough to show:

\smallskip

\begin{Proposition}\label{converse}  If  a function $\u:\T^N\to\R^M$ satisfies  inequality \eqref{formusub}   for
any pair of points $x$, $y$  in $\T^N$, $\a \in \S$, any bounded
stopping time $\tau$,   $\Xi \in \K(\tau, y-x)$, then $\u$ is a
subsolution to \eqref{HJa}.
\end{Proposition}

\begin{proof} By using the same argument of Lemma \ref{subbolip}, we
see that $\u$ is Lipschitz--continuous.  Fix $i \in \unoM$ and take
a differentiability point  $y$ of $u_i$,  define $\v$ as in
\eqref{formula} with $\u(y)$ in place of $\b$, then, owing to
Theorem \ref{vai}
\[
   \v \geq \u \quad\hbox{in $\T^N$} \quad\hbox{and} \quad \v(y)= \u(y).
\]
Hence $u_i$ is subtangent to $v_i$ at $y$, which implies  $Du_i(y)
\in
\partial v_i(y)$ and, being $\v$ subsolution to \eqref{HJa},  by Theorem
\ref{mainsub} and Remark \ref{solu} we get

\[ H_i(y,Du_i(y)) + \La^i \, \u(y) = H_i(y,Du_i(y)) + \La^i \, \v(y)
\leq \al.\] This concludes the proof.
\end{proof}

\bigskip

\appendix

\section{Stochastic matrices} \label{uno}

 \parskip +3pt

 In this appendix we collect some basic material on stochastic
 matrices. All matrices appearing below are square matrices. We
 refer to \cite{Me,No}   for the results stated without proof.

 \medskip

 We denote by $\S \subset \R^M$  the simplex of {\em probability
 vectors} of $\R^M$, namely with nonnegative components summing to $1$.

 \medskip

 \begin{Definition} A (right)  stochastic matrix is a matrix
possessing  nonnegative entries and with each row  summing to $1$.
\end{Definition}

\smallskip

\begin{Proposition}\label{stocha1} A matrix $B$ is stochastic if and only
\begin{equation}\label{perron}
    \a \, B \in \S \;\; \hbox{whenever $\a \in \S$}.
\end{equation}
\end{Proposition}
\begin{proof}   $B$ is stochastic if and only if all its rows
are probability vectors, or, in other terms, if and only if
\[ \ei \cdot B \in \S \qquad\hbox{for any $i$}.\]
 this is in turn equivalent to \eqref{perron}.
\end{proof}

\medskip

By the Perron-Frobenius theorem  for nonnegative matrices,  we have

\begin{Proposition}\label{stocha02}  Let $B$ be a stochastic matrix, then its
maximal  eigenvalue is $1$ and there is a corresponding left
eigenvector in $\S$.
\end{Proposition}

\medskip

By the Perron-Frobenius theorem for positive matrices,  we have
\begin{Proposition}\label{stocha2}  Let $B$ be a positive stochastic matrix, then its
maximal  eigenvalue is $1$ and is simple.  In addition, the  unique
corresponding  left eigenvector belonging to $\S$  is positive.
\end{Proposition}

\medskip

Even if it is an elementary fact, we give for completeness the proof
of the key property that the coupling matrix of the Hamilton--Jacobi
system under  investigation spans a semigroup of stochastic
matrices.

\smallskip

\begin{Proposition}\label{stochazz}
For a  matrix $A$,   $e^{- t A}$ is stochastic for any $t$, if and
only if  {\bf(H4)}, {\bf(H5)} hold with $A$ in place of $\La$.
\end{Proposition}
\begin{proof}
Assume that $A$ satisfies {\bf(H4)}, {\bf(H5)},  then, given $t
>0$, $I - \frac {t A} n$ is stochastic for $n$ suitably large,
consequently $\left ( I - \frac {t A} n \right ) ^n$ is stochastic
because the product of stochastic  matrices is still stochastic, and
\[ e^{-  tA } = \lim_{n\to\infty} \left ( I -
\frac {t A} n \right ) ^n\] is stochastic because stochastic
matrices make up a  compact subset in the space of square matrices.
Conversely, if $e^{- t A}$ is stochastic then the relation
\[ A = \lim_{t \to 0} \frac{I - e^{- t A}}t\]
implies that $A$ satisfies {\bf(H4)}, {\bf(H5)}.

\end{proof}


\begin{Proposition}\label{liplip}  The function
 \[ t \mapsto \ei \, e^{- t\La}\]
 is Lipschitz continuous  in $[0,+ \infty)$ for any $i \in \unoM$.
\end{Proposition}
\begin{proof}
We have
\[ \frac d{dt} \ei \, e^{- t\La} = - \ei \, \La \, e^{- t\La}\]
which is bounded in $t \in [0,+ \infty)$ because the matrices
$e^{-t\La}$,  being stochastic, vary in a compact subset of the
space of $M \times M$ matrices.
\end{proof}

\bigskip

\section{Path spaces}\label{due}
We refer to \cite{Bi} for more details in this section.
The term {\em cadlag} corresponds to the French acronym {\em continu
\`{a} droite limite \`{a} gauche}, namely continuous on the right
and with left limit. We consider the space of cadlag paths defined
in $[0,+\infty)$,  with value in $\unoM$ and $\R^N$, denoted by
$\D:=\DiM$ and $\DiR$, respectively. For any $t >0$, we also
indicate by $\D(0,t;\unoM)$ the space of cadlag paths defined in
$[0,t]$ with values in $\unoM$. It can be proved that
\begin{eqnarray}
 \label{ciccio1}   &\hbox{Any cadlag path has at most countably many
    discontinuities.}&  \\
 \label{ciccio2}  &\hbox{Any cadlag path is locally (in time) bounded.}&
\end{eqnarray}

\medskip

\begin{Terminology} \label{cylinder}

To any finite increasing sequence of times $t_1, \cdots, t_k$, with
$k \in \N$,  and indices $j_1, \cdots, j_k$ in $\unoM$ we associate with
a (thin) {\em cylinder} defined as
\begin{equation}\label{cyl0}
   \mathcal C(t_1, \cdots,t_k;j_1, \cdots,j_k)= \{\om \mid
\om(t_1)=j_1, \cdots, \om(t_k)=j_k\} \subset \D.
\end{equation}
To ease notations, we set
\begin{equation}\label{defDi}
    \D_i = \mathcal C(0;i) \qquad\hbox{for any $ i \in \unoM$.}
\end{equation}
 We call {\em multi-cylinders} the sets
made up by finite unions of mutually disjoint cylinders.

\end{Terminology}

\medskip

We  endow $\mathcal D$ with the $\sigma$--algebra $\F$ spanned by
cylinders, those of the type $\mathcal C(s;j)$ for $s \geq 0$, $j
\in \unoM$ are indeed enough. A natural related filtration $\F_t$ is
obtained by picking, as generating sets, just the cylinders
$\mathcal C(t_1, \cdots,t_k;j_1, \cdots,j_k)$ with $t_k \leq t$, for
any fixed $t\geq 0$.

\medskip

Same construction, {\em mutatis mutandis} can be performed in
$\DiR$, in this case the $\sigma$--algebra is spanned by cylinders
of the type
\[\{\xi \in \DiR \mid \xi(s) \in E\}\]
for $s$, $E$ varying in $[0,+\infty)$ and in the Borel
$\sigma$--algebra related to the natural topology of $\R^N$,
respectively.

\medskip

Both $\D$ and $\DiR$ can be endowed with a metric, named after
Skorohod, which make them Polish spaces, namely complete and
separable, and such that the aforementioned $\sigma$--algebras are
the corresponding Borel $\sigma$--algebras

\smallskip
\begin{Remark}\label{projeproje} A consequence of the previous
definitions is  that  $\F$ is the minimal $\sigma$--algebra for
which the evaluation maps
\[ t \mapsto \om(t) \qquad t \in [0,+ \infty)\]
are measurable and the same holds true for the $\sigma$--algebra in
$\DiR$ with respect to the evaluation maps
\[\xi \mapsto \xi(t).\]
A map $\Xi: \D \to \DiR$ (resp $\phi: \D \to \D$) is accordingly
measurable if and only if the maps $\om \mapsto \Xi(\om)(t)$ from
$\D$ to $\R^N$ (resp., $\om \mapsto \phi(\om)(t)$ from $\D$ to
$\unoM$) are measurable for any $t$.
\end{Remark}

\medskip

The convergence induced by Skorohod metric can be defined, say in
$\DiR$ to fix ideas,  requiring  that there exists a sequence $g_n$
of of   increasing continuous functions from $[0, + \infty)$  onto
itself (then $g_n(0)=0$ for any $n$) such that
\begin{eqnarray*}
  g_n(s) &\to& s \quad\hbox{uniformly in $[0,+ \infty)$} \\
  \xi_n(g_n(s)) &\to& \xi(s) \quad\hbox{locally uniformly in $[0, + \infty)$.}
\end{eqnarray*}

This is basically locally uniform convergence, up to an uniformly
small deformation of the time scale given by the $g_n$.  We infer
from the previous definition  that
\begin{eqnarray}
 \label{ciccio21}   &\xi_n \to \xi \;\;\hbox
 {in the Skorohod sense} \; \Rightarrow \xi_n(t) \to \xi(t) \;\;\hbox{at any
   continuity
point of $\xi$} &
\end{eqnarray}
which in particular implies
\begin{eqnarray}\label{ciccio222} &\xi_n
\to \xi \;\;\hbox
 {in the Skorohod sense} \; \Rightarrow \xi_n(0) \to \xi(0) &
\end{eqnarray}

 We  moreover have
\begin{eqnarray}
 \label{ciccio3}  &\hbox{Any  sequence  convergent in the Skorohod
 sense is locally uniformly bounded.}&
\end{eqnarray}
\medskip

For $t >0$, we  say that a path in  $\DiR$ is  {\em piecewise
constant} in $[0,t]$  if is  of the form
\[ \sum_{k=1}^{l-1} x_k \, \I([s_k,s_{k+1})) \qquad\hbox{for $s \in [0,t)$}\]
where $x_k \in \R^N$ and $s_k$ is an increasing sequence of times
with $s_1=0$, $s_l=t$. We will use the following approximation
result, see \cite{Bi} Section 12, Lemma 3, in a  version, slightly
accommodated to our needs.

\smallskip

\begin{Proposition}\label{piece1} For $t>0$ and  $\xi \in \DiR$,
let $s^n_k$, $k =1, \cdots, l_n$,   be a family of strictly
increasing finite sequences  with $s^n_1=0$, $s^n_{l_n}=t$ and
\[ \sup_k s_{k}^n - s_{k-1}^n \to 0  \qquad\hbox{as $n$ goes to
infinity}\] then the sequence of (piecewise constant in $[0,t]$ )
paths
\[ \xi_n= \left \{  \begin{array}{ll}
          \sum_k \xi(s^n_k) \, \I([s^n_{k-1},s^n_{k})) & \hbox{in \ $[0,t)$ }\\
           \xi &   \hbox{in \ $[t, + \infty)$}
         \end{array} \right . \]   converges to $\xi$ in $\DiR$.
\end{Proposition}

\bigskip

For any $h >0$, we consider the shift flow $\phi_h$ on $\mathcal D$
defined by
\begin{equation}\label{flow}
    \phi_h(\om)(s)= \om(s+h) \qquad\hbox{for any $s \in [0,+
\infty)$, $\om \in \mathcal D$.}
\end{equation}

\smallskip
Notice that $\phi_h$ is not in general continuous since the fact
that $\om_n \to \om$ in the Skorohod metric does not in general
implies that $\phi_h(\om_n)(0)=\om_n(h) \to \phi_h(\om)(0)= \om(h)$,
unless of course $h$ is a continuity point for $\om$,  and so does
not in turn implies, by \eqref{ciccio222}, that $\phi_h(\om_n)$
converges to $\phi_h(\om)$. However we directly derive from Remark
\ref{projeproje}:

\begin{Proposition}\label{soloshift} The shift flow $\phi_h: \D \to
\D$ is measurable for any $h >0$.
\end{Proposition}

\smallskip

\begin{Proposition}\label{supershift} For nonnegative constants $h$,
$t$, we have
\[ \phi_h^{-1} (\F_t) \subset \F_{t+h}.\]
\end{Proposition}
\begin{proof} For any $t_1 \geq 0$, $j_1 \in \unoM$ we have
\[ \phi_h^{-1} (\mathcal C(t_1;j_1)) = \mathcal C(t_1+h, j_1).\]
The assertion thus comes from the fact that $\F_t$ is spanned by
cylinders of the form $\mathcal C(t_1;j_1)$, with $t_1 \leq t$, and
in this case $\mathcal C(t_1+ h;j_1) \in \F_{t+h}$.
\end{proof}

\medskip

We also consider that space $\CiT$ of continuous paths defined in
$[0,+ \infty)$ taking values in $\T^N$. It is endowed with a metric
giving it the structure of a Polish space, which induces the  local
uniform convergence.

We define a map
\[X: \DiR \to \CiT\]
via
\begin{equation}\label{proje}
   \mathcal I(\xi)(t) = \proj \left ( \int_0^t \xi \, ds \right ).
\end{equation}
where $\proj$ indicates the projection from $\R^N$ onto $\T^N$.

\begin{Proposition}\label{contcontprop}
The map $\mathcal I(\cdot)$ is continuous.
\end{Proposition}
\begin{proof} Let us consider  a sequence $\xi_n$ in $\DiR$
converging to some $\xi$, then  by \eqref{ciccio3} it is locally (in
time) uniformly bounded and by \eqref{ciccio1}, \eqref{ciccio21}
\[ \xi_n(s) \to \xi(s) \qquad\hbox{a.e. in $[0,+ \infty)$.} \]
 Then by  the dominated convergence theorem and continuity of $\proj$
 \begin{equation}\label{contcont}
    \mathcal I(\xi_n)(t) \to \mathcal I(\xi)(t) \qquad\hbox{for any $t$}.
\end{equation}
Furthermore, from the uniformly boundedness of $\xi_n$   and the fact that $\proj$
is nonexpansive, we derive that the $\mathcal I(\xi_n)$ are locally
equiLipschitz continuous and locally uniformly bounded. By the Arzel\`a-Ascoli theorem with \eqref{contcont}, we get
\[ \mathcal I(\xi_n) \to \mathcal I(\xi) \qquad\hbox{locally uniformly in time,}\]
as desired
\end{proof}

\medskip

For $\om \in \D$, $t >0$, $x \in \R^N$, we consider the function
\begin{equation}\label{defJ}
    \xi \mapsto \int_0^t L_{\om(s)}(x + \mathcal I(\xi)(s), - \xi(s)) \, ds
\end{equation}
from $\DiR$ to $\R$.

\smallskip

\begin{Proposition}\label{contJ} The function defined in \eqref{defJ} is
continuous.
\end{Proposition}
\begin{proof} Let $\xi_n$ be a sequence converging to some $\xi$ in
$\DiR$, then the $\xi_n$ are uniformly bounded in $[0,t]$ and converge
pointwise to $\xi$ a.e by \eqref{ciccio1}, \eqref{ciccio21},
\eqref{ciccio3}. Furthermore, bearing in mind Proposition
\ref{contcontprop}, we know that $\mathcal I(\xi_n)$ converges to
$\mathcal I(\xi)$ in $\CiT$. Using the continuity of $L_i$, for any
$i$, we derive that
\[L_{\om(s)}(x + \mathcal I(\xi_n), - \xi_n) \to L_{\om(s)}(x + \mathcal I(\xi), -
\xi) \quad\hbox{a.e. in $[0,t]$}\]  and, in addition, that the
$L_{\om(s)}(x + \mathcal I(\xi_n), - \xi_n)$ are uniformly bounded. We
thus get the assertion through the dominated convergence theorem.
\end{proof}

%

\end{document}